\documentclass[12pt]{amsart}
 \usepackage{mathtools}
\mathtoolsset{showonlyrefs,showmanualtags}

\usepackage[backrefs]{amsrefs}

 \usepackage[top=1.5in, bottom=1.5in, left=1.25in, right=1.25in]	{geometry}
\usepackage[all]{xy}
\usepackage{amsmath}
\usepackage{amssymb}
\usepackage{amsthm}
\usepackage{amscd}

\newcommand{\R}{\mathcal{R}}

\newcommand{\F}{\mathcal{F}}

\newcommand{\RR}{\mathbb  R}
\newcommand{\TT}{\mathbb  T}

\newcommand{\Z}{\mathbb  Z}

\numberwithin{equation}{section}

\newtheorem{theorem}[equation]{Theorem}
\newtheorem{definition}[equation]{Definition}
\newtheorem{proposition}[equation]{Proposition}
\newtheorem{cor}[equation]{Corollary}
\newtheorem{lemma}[equation]{Lemma}
\newtheorem{problem}[equation]{Problem}
\newtheorem{remark}[equation]{Remark}
\newtheorem{claim}[equation]{Claim}

\usepackage{hyperref} 
\hypersetup{
    colorlinks=true,       
    linkcolor=blue,          
    citecolor=magenta,        
    filecolor=magenta,      
    urlcolor=cyan           
}

\begin{document}
\title[Discrete Directional Maximal Functions]{Discrete Analogues in Harmonic Analysis: Directional Maximal Functions in $\mathbb{Z}^2$}
\author{Laura Cladek}
\address{Department of Mathematics, UCLA\\
520 Portola Plaza, Los Angeles, CA 90095
}
\email{cladekl@ucla.edu}
\date{\today}

\author{Ben Krause}
\address{
Department of Mathematics,
Caltech \\
Pasadena, CA 91125}
\email{benkrause2323@gmail.com}
\date{\today}

\begin{abstract}
Let $V = \{ v_1,\dots,v_N\}\subset\mathbb{Z}^2$ be a collection of $N$ vectors that live near a discrete sphere. We consider discrete directional maximal functions on $\mathbb{Z}^2$ where the set of directions lies in $V$, given by
\[ \sup_{v \in V, k \geq C \log N} \left| \sum_{n \in \mathbb{Z}} f(x-nv ) \phi_k(n) \right|, \ f:\mathbb{Z}^2 \to \mathbb{C}, \]
where $\phi_k(t) := 2^{-k} \phi(2^{-k} t)$ for some bump function $\phi$. Interestingly, the study of these operators leads one to consider an ``arithmetic version'' of a Kakeya-type problem in the plane, which we approach using a combination of geometric and number-theoretic methods. Motivated by the Furstenberg problem from geometric measure theory, we also consider a discrete directional maximal operator along polynomial orbits,
\[ \sup_{v \in V} \left| \sum_{n \in \mathbb{Z}} f(x-v\cdot P(n) ) \cdot \phi(n) \right|, \ P \in \mathbb{Z}[-].  \]
\end{abstract}

\maketitle

 \setcounter{tocdepth}{1}
\tableofcontents 

\section{Introduction}

Discrete analogues of (continuous) operators in harmonic analysis has been an active area of research since Bourgain initiated their study in the course of his work on pointwise ergodic theorems in the late 80s and early 90s, \cite{B1,B2,B3}. In particular, motivated by pointwise ergodic theoretic concerns (cf.\ Calder\'{o}n's transference principle \cite{C}), the initial focus of discrete harmonic analysis was in understanding maximal averaging operators along polynomial orbits, say
\begin{equation}\label{e:Zsquares}
\sup_k \left| \frac{1}{2^k} \sum_{n \leq 2^k} f(x-n^2) \right|, \ f : \mathbb{Z} \to \mathbb{C}.
\end{equation}

Although the analogous Euclidean operator,
\begin{equation}\label{e:Rsquares}
\sup_k \left| \frac{1}{2^k} \int_{0}^{2^k} f(x-t^2) \right|, \ f : \mathbb{R} \to \mathbb{C}
\end{equation}
is governed by the (continuous) Hardy-Littlewood maximal operator by a change of variables, there is no such connection between \eqref{e:Zsquares} and the discrete Hardy-Littlewood maximal function: averaging over the set of the squares
\[ \{ n^2 : n \geq 1\} \]
is more akin -- from a density perspective -- to a continuous averaging operator over a set of lower dimension. Although developing an appropriate $\ell^p$ theory for these averages, or for the related singular integral formulation, required delicate analysis, necessitated by arithmetic complications unique to the discrete setting, the theory of polynomial Radon transforms is essentially complete: aside from Bourgain's work, major contributions to the field were made by Magyar, Stein, and Wainger \cite{MSW}, by Ionescu and Wainger \cite{IW}, and more recently by Mirek, Stein, and Trojan \cite{MST1,MST2}, and by Mirek, Stein, and Zorin-Kranich \cite{MSZK}.

In light of the efforts of the above authors and others, the field has developed sufficiently robust tools that discrete analogues of more complicated questions can be meaningfully posed: in \cite{KL} and \cite{BK}, the second author, partially in collaboration with Lacey, initiated a study into discrete analogues of maximally modulated oscillatory singular integrals, of the type considered by Stein \cite{S} and Stein-Wainger \cite{SW}; in \cite{KesLac}, \cite{Kes}, and \cite{KLM}, Kesler-Lacey, Kesler, and Kesler-Lacey-Menas establish $\ell^p$-improving (and sparse) estimates for spherical averages and maximal functions are established, in analogy with the work of Schlag \cite{Sch}.

The purpose of this paper is to begin an investigation into discrete analogues of directional maximal functions in the plane; we provide a brief summary of the ($L^2$-) planar theory, and refer the reader to the recent preprint of Di Plinio and Parissis \cite{DpP} for a more comprehensive discussion.

The initial interest interest in the directional maximal function in the plane,
\[ M_V f(x) := \sup_{v \in V \subset \mathbb{S}^1} \left| \int_0^1 f(x - t\cdot v) \ dt \right|, \ f : \mathbb{R}^2 \to \mathbb{C}, \]
is due to its connection with the Kakeya maximal function over $1/N \times 1$ tubes when $|V| = N$; in particular, when $V$ is uniformly distributed, the two maximal functions both have operator norm on $L^2$ given by a universal constant times $\log^{1/2} N$, due to \cite{Cor} (sharpness follows from the existence of Kakeya sets in the plane). For general sets of directions, $V \subset \mathbb{S}^1$, the bound of $\log^{1/2} N$ was proven by Katz in \cite{Katz}, and re-proven by Demeter in \cite{D}; we begin our study of the discrete directional maximal functions with the aim of understanding the (single-scale) maximal average in the plane,
\begin{equation}\label{e:Aphi}
A_{V,\phi_k}f(x) := \sup_{v \in V} \left| \sum_{n \in \mathbb{Z} } f(x -nv) \phi_k(n) \right| 
\end{equation}
where $\phi$ is some bump function (say) on the real line, $\phi_k(x) := 2^{-k} \phi(2^{-k} x)$ are the usual $L^1$-normalized dilations, and for some fixed $A$,
\[ V \subset \{ |x| \approx A, x \in \mathbb{Z}^2 \},\]
is a collection of $|V| = N$ vectors.\footnote{In stricter analogy with the Euclidean case, one might expect to be able to select vectors on a (discrete) circle; the irregular distribution of lattice points introduces significant technical complications.}
(Because our estimates will be uniform over appropriately normalized bump functions, $\phi$, for notational ease we will often suppress the dependence of our maximal operators on these functions, writing, for instance, $A_{V,k} = A_{V,\phi_k}$.) In particular, we are after the $\ell^2$ operator norm of \eqref{e:Aphi}. Upon first consideration, this question is not very interesting: by considering bumps adapted to unit scales, we see that a sharp estimate for $k$ close to $1$ is given by $N^{1/2}$. However, unlike the Euclidean setting, this problem is not scale-invariant, which leads to the natural question:
\begin{problem}\label{p:indep}
For collections of vectors $V$ of cardinality $N$, can one find a scale $k = k(N)$ beyond which
\begin{equation}\label{e:MPop}
\| A_{V,k} \|_{\ell^2(\mathbb{Z}^2) \to \ell^2(\mathbb{Z}^2)}
\end{equation}
is essentially independent of $|V| = N$? 
\end{problem}

Although the continuous analogue of this problem could be (essentially) answered via a $TT^*$ argument, see \cite{Ba}, neither this approach, nor the more delicate of analysis of Katz, is appropriate in the discrete setting, where the geometry of the problem becomes radically different; for instance, in general, non-parallel lines 
\[ \{ m + n v : n \in \mathbb{Z} \}, \ \{ m' + n v' : n \in \mathbb{Z} \} \subset \mathbb{Z}^2, \ v \neq v'  \]
will \emph{not} intersect. Passing to Fourier space connects these geometric issues with arithmetic ones, and our answer to Problem \ref{p:indep} requires an interplay between these two perspectives: although a universal answer to Problem \ref{p:indep} seems out of reach of current techniques, we are able to construct sets of vectors for which we can answer this question quite thoroughly. In particular, we will show that for each $\epsilon > 0$ and $N \geq 1$, there exist collections of vectors $V = V_{N,\epsilon}$ so that $V$ lives near a discrete sphere, has $|V| = N$, and so that $A_{V,k}$ has $\ell^2(\mathbb{Z}^2)$ operator norm bounded by a constant multiple of $|V|^{\epsilon}= N^{\epsilon}$:
\begin{theorem}\label{linthm}
Let $\epsilon >0$ and $N>0$. Then there exists absolute constants $C_\epsilon$ and $C_0=C_0(\epsilon)$ so that the following holds:

There exists a collection $V=V_{N, \epsilon}\subset\mathbb{Z}^2$ of vectors satisfying $|V|=N$, $V\subset\{z\in\mathbb{Z}^2:\,|z|\approx A\}$ for any $A$ satisfying $A\ge N^{C_0}$, and we have the estimate
\begin{equation}\label{e:singscaleest}
\|A_{V, k}\|_{\ell^2(\mathbb{Z}^2)\to\ell^2(\mathbb{Z}^2)}\leq C_\epsilon \cdot N^{\epsilon}
\end{equation}
provided that $2^k \geq N^{C_1}$ for some $C_1=C_1(A)$ sufficiently large.
\end{theorem}
In fact, this construction is sufficiently robust to allow for a supremum over scales to be introduced.
To this end, define
\begin{equation}\label{e:scales-max}
    A^*_{V}f(x):= \sup_{2^k \geq N^{C_1}, v \in V} \left| \sum_n f(x-nv)\phi_k(n) \right|
\end{equation}
\begin{cor}\label{cor:multiscalemax}
In the setting of Theorem \ref{linthm}, \eqref{e:singscaleest} is satisfied by $A^*_{V}$ as well.
\end{cor}

\subsection{Discrete Maximal Functions along Sparse Sequences}
We will also study discrete analogues where the geometry of the corresponding Euclidean problem in the plane is much less understood. While the Kakeya conjecture in the plane is fully resolved, very little progress has been made toward understanding the related Furstenberg conjecture in the plane. 
\newline
\indent

Recall that a \textit{Furstenberg set with parameter $\beta$} in the plane is a compact set $K\subset\mathbb{R}^2$ such that for every direction $w\in \mathbb{S}^1$, there exists a line segment, $l_w$, pointed in direction $w$, so that
\begin{equation}\label{e:Haus}
    \text{dim}_{H} (K \cap l_w) \geq \beta
\end{equation}  
for all $w \in \mathbb{S}^1$; here $\text{dim}_{H}$ refers to Hausdorff dimension. More generally, one may consider two-parameter $(\alpha,\beta)$ Furstenberg sets to be subsets $K$ of the plane so that \eqref{e:Haus} holds for $w \in V \subset \mathbb{S}^1$, for some subset $V \subset \mathbb{S}^1$ of Hausdorff dimension $\geq \alpha$. Thus, in analogy to the relationship between the Kakeya-Nikodym maximal operator and Kakeya sets, one would be motivated to introduce a Furstenberg-type maximal operator
\begin{equation}\label{furst}
M_V^{\beta}(f)(x)=\sup_{v\in V} \left| \mu_v\ast f(x) \right|, 
\end{equation}
where $V$ is a subset of $\mathbb{S}^1$ of Hausdorff measure $\alpha$ and $\mu_v$ is the Hausdorff measure of a compact $\beta$-dimensional subset of the real line, embedded into the plane and rotated to be parallel with $v$. The study of Furstenberg sets has a rich history in geometric measure theory, and is deeply connected to the Falconer and sum-product conjectures, as explored, for instance, in \cite{KaTao}.
\newline
\indent
In the discrete setting, we consider the maximal function,
\begin{equation}\label{e:avgsquares}
A^{2}_{V,k}f(x) := \sup_{v \in V} \left| \sum_{n} f(x - n^2v) \phi_k(n) \right|.
\end{equation}
By density considerations, restricting integer sequences to squares may be viewed as the introduction of Radon-type behavior. Here, the counterpart to
\begin{equation}\label{e:linFT}
 \sum_n e^{-2\pi i n (v \cdot \beta) } \cdot \phi_k(n) ``=" \mathbf{1}_{\| v \cdot \beta\| \leq 2^{-k}  },
\end{equation}
where $\|x\|$ denotes the distance from $x$ to the nearest integer, is 
\begin{equation}\label{e:inc1}
 \sum_{n} e^{-2\pi i n^2 (v \cdot \beta) } \cdot \phi_k(n) ``=" \sum_{q\leq C_\epsilon \cdot 2^{\epsilon k}, \ (a,q) = 1} \mathbf{1}_{\|v \cdot \beta - a/q\| \leq 2^{-2k}  },  
\end{equation}
(see \S \ref{ss:app} below for a precise statement); in particular, multi-frequency considerations are paramount.

Nevertheless, provided our sets of directions are given by $V=V_{N,\epsilon}$ as in the statement of Theorem \ref{linthm}, we are again able to bound this maximal function with weak dependence on $|V|$.
\begin{theorem}\label{t:2avg}
Let $\epsilon >0$ and $N>0$. There exists absolute constants $C_\epsilon, C_0=C_0(\epsilon)>0$, so that the following holds:

There exists a set of vectors $V = V_{N,\epsilon}$ (the same collection as in Theorem \ref{linthm}) with $|V|=N$ and $V\subset\{x\in\mathbb{Z}^2:\,|x|\approx A\}$ for any $A\ge N^{C_0}$ so that, if $k \geq C_{2} \log N$ for $C_2=C_2(A)$ sufficiently large, then
\begin{equation}\label{e:maxpbound}
\| A^{2}_{V,k} \|_{\ell^2(\mathbb{Z}^2) \to \ell^2(\mathbb{Z}^2) } \leq C_\epsilon \cdot N^{\epsilon}.
\end{equation}
\end{theorem}

In fact, one may replace the monomial sequence $\{ n^2 : n \geq 1\}$ with the image of any polynomial with integer coefficients:
with
\[ A^{P}_{V,k} f(x):= \sup_{v \in V} \left| \sum_{n} f(x - P(n)v ) \phi_k(n) \right|\]
we have the following theorem.

\begin{theorem}\label{t:avg}
Let $\epsilon >0$ and $N>0$. There exists absolute constants $C_\epsilon, C_d=C_d(\epsilon)>0$ so that the following holds:

For any $N$, there exists a set of vectors $V = V_{N,\epsilon}$ (the same collection as in Theorem \ref{linthm}) satisfying $|V|=N$ and $V\subset\{x\in\mathbb{Z}^2:\,|x|\approx A\}$ for any $A\ge N^{C_d}$, so that, if $k \geq C_d' \log N$ for some $C_d'=C_d'(A)$ sufficiently large and $P$ is any polynomial of degree $d$, then
\[ 
\| A^{P}_{V,k} f \|_{\ell^2(\mathbb{Z}^2) \to \ell^2(\mathbb{Z}^2)} 
\leq C_\epsilon \cdot N^{\epsilon}.
 \]
\end{theorem}


The key issue is the distribution of the sequence
\[ \{ P(n) \mod q: n \} \]
for each integer $q \geq 1$; by a classical result of Hua, this uniform distribution can be quite precisely quantified in terms of decay of various exponential sums, cf.\ \cite[\S 7, Theorem 10.1]{Hua}, and we are able to conclude our result. With this heuristic in mind, that we are able to conclude analogous results for (polynomial images, or even ``thin'' images \cite{M}, of) the primes, or for randomly generated sequences of prescribed density of the type introduced by Bourgain in \cite[\S 8]{B1},
should be somewhat expected.

Finally, we remark that in the foregoing discussion, all results are remain valid upon replacing our bump functions, $\phi$, with appropriately dilated (one-dimensional) Calder\'{o}n-Zygmund kernels, $K$, see \cite{D} for the continuous theory. As motivation for this perspective, we note that the discrete analogue of Carleson's theorem ($d=1$), and the main result of \cite{BK} $(d\geq 2)$ exhibit the $\ell^2$ boundedness of
\[ \sum_m \frac{ f(x-m,y-v(x) \cdot m^d) }{m}\]
for any $v : \mathbb{Z} \to \mathbb{Z}$; this can be seen by taking a partial Fourier transform in the $y$ variable. The content of our current theorem is the following.

\begin{theorem}\label{t:sing}
For any $\epsilon > 0$ and any polynomial of degree $d$, there exists an absolute constant $C_{\epsilon,d}$ so that one may find $v = (v_1,v_2) : \mathbb{Z}^2 \to V$ with $|V| = N$ arbitrary large, so that
\[ \left\| \sum_{|m| \geq N^{C_{\epsilon,d}}} \frac{ f(x-v_1(x,y)\cdot m,y-v_2(x,y) \cdot P(m)) }{m} \right\|_{\ell^2(\mathbb{Z}^2)} \leq C_{\epsilon,d} \cdot N^{\epsilon} \cdot \|f\|_{\ell^2(\mathbb{Z}^2)}.\]
\end{theorem}

\bigskip

The structure of the paper is as follows:

In \S \ref{s:inc} we prove our multi-frequency incidence estimates; 

In \S \ref{s:easyproof}, we apply these incidence estimates to quickly prove Theorem \ref{linthm} (and Corollary \ref{cor:multiscalemax});

In \S \ref{s:squares} we combine the above analysis with further number-theoretic techniques from discrete harmonic analysis to prove Theorem \ref{t:2avg} (and Theorem \ref{t:avg}).


\subsection{Acknowledgements}
The authors would like to thank Francesco di Plinio for sparking their interest in directional maximal functions. They would also like to thank Michael Lacey for his continued support and help in the editing process, and to Polona Durcik and Jos\'{e} Madrid for help in the editing process and many useful comments that significantly improved the clarity of the final draft.

\subsection{Notation}
Here and throughout, $e(t) := e^{2\pi i t}$; $x \equiv y$ will denote equivalence $\mod 1$. Throughout, $C$ will be a large number which may change from line to line. 

For finitely supported functions on $\mathbb{Z}^d, \ d=1,2$, we define the Fourier transform
\[ \F_{\mathbb{Z}} f(\beta) := \hat{f}(\beta) := \sum_n f(n) \, e(-\beta n),\]
with inverse
\[ \F_{\mathbb{Z}}^{-1}g(n) := g^{\vee}(n) := \int g(\beta) \, e(\beta n) \ d\beta,\]
where integration occurs over $\mathbb{T}$ or $\mathbb{T}^2$ depending on dimension.

We will let $\chi$ be an even non-negative compactly supported Schwartz functions which approximates the indicator function of an interval centered at the origin.
\begin{equation}\label{chi}
\mathbf{1}_{|\xi| \leq c} \leq \chi \leq \mathbf{1}_{|\xi| \leq 2c}.
\end{equation}

We will use
\[ \| f \|_{\ell^p} := \left( \sum_{x \in \mathbb{Z}^2} |f(x)|^p \right)^{1/p} \]
with the obvious modification at $p=\infty$. When we sum over $\mathbb{Z}$ we will specify the domain explicitly.

We will make use of the modified Vinogradov notation. We use $X \lesssim Y$, or $Y \gtrsim X$, to denote the estimate $X \leq CY$ for an absolute constant $C$. We use $X \approx Y$ as shorthand for $Y \lesssim X \lesssim Y$. We also make use of big-O notation: we let $O(Y )$ denote a quantity that is $\lesssim Y$. If we need $C$ to depend on a parameter, we shall indicate this by subscripts, thus for instance $X \lesssim_p Y$ denotes the estimate $X \leq C_p Y$
for some $C_p$ depending on $p$. We analogously define $O_p(Y)$.

\section{A Multi-Frequency Incidence Problem}\label{s:inc}
Before turning to the main result of this section, we briefly pause to explain the nature of our multi-frequency incidence problem.

The departure point is that an upper bound for the ``arithmetic" directional maximal function
\begin{equation}\label{e:opNorm}
\| A_{V,k} \|_{\ell^2(\mathbb{Z}^2) \to \ell^2(\mathbb{Z}^2)}
\end{equation}
 is -- up to standard transference arguments -- essentially given by
\begin{equation}\label{e:inc0}
\sup_{ |\beta| \gtrsim A^{-2}} \left( \sum_{v \in V} \sum_{m \in \mathbb{Z}} \mathbf{1}_{|v \cdot \beta - m | \ll 2^{-k}}(\beta) \right)^{1/2};
\end{equation}
here $V \subset \{ |x| \approx A, \ x \in \mathbb{Z}^2 \}$, and $2^k \geq N^{C_1}$ is sufficiently large.

Note that the quantity \eqref{e:inc0} is the $L^{\infty}$ norm outside of a ball of radius $A^{-2}$ of a sum of characteristic functions of tubes of thickness $\ll 2^{-k}A^{-1}$ that are perpendicular to $v$ for some $v\in V$, where the tubes pointing in a given direction are spaced apart in an arithmetic progression with spacing $\approx A^{-1}$. (Note that it is necessary to remove a ball at the origin since for \textit{any} collection of vectors $V$, the quantity inside the supremum in \eqref{e:inc0} is $\approx N^{1/2}$ near the origin.) Interestingly, a heuristic that emerges both in this linear setting, and then again later in the polynomial setting, where multi-frequency complications arise, is that those collections of vectors that are amenable to obtaining $\ell^2(\Z^2)$ bounds are those whose (normalized) coordinates exhibit an \emph{intermediate} amount of arithmetic independence in terms of prime factorization. With these remarks in mind, we turn to our incidence problem. 

\bigskip

Recall that we have fixed $\epsilon, N>0$. For a given collection of vectors $V$ and for each $v\in V$, consider the union of tubes

\begin{equation}
K_{k, v}^A:=\{\beta\in\mathbb{T}^2, |\beta|\gtrsim A^{-2}:\,|v\cdot\beta-m|\lesssim 2^{-k}A^{-1}\text{ for some }m\in\mathbb{Z}\}.
\end{equation}
To bound \eqref{e:inc0}, we would thus like to estimate
\begin{equation}\label{in}
\bigg\|\sum_{v\in V}{\bf 1}_{K_{k, v}^A}\bigg\|_{L^{\infty}(\mathbb{T}^2\setminus B_{A^{-2}}(0))}.
\end{equation}
However, we instead choose to estimate the following more general quantity, which we will see arises in our study of the directional maximal function along the squares (and more generally polynomial orbits) in a manner analogous to \eqref{e:inc0}, and which in fact dominates (\ref{in}) for an appropriate choice of parameters. Fix an integer $s\ge 1$ with $N^{\epsilon}<2^s<N^{1/\epsilon}$ and $C_0>0$ an absolute constant assumed to be sufficiently large. For $2^{s-1} \leq r < 2^s$, and $v \in V$, consider the union of tubes,
\begin{equation}\label{e:Kr}
K_{r,s, v}^A := \{ \beta \in \mathbb{T}^2, \ |\beta| \gtrsim A^{-2} : |v \cdot \beta - m - b/r| \lesssim 2^{-C_1 s} \text{ for some } b \leq r, m \in \mathbb{Z} \}.
\end{equation}
Note that for the choice of $s$ such that $2^s=N$ and for $2^k\ge N^{C_1}\ge N^{C_0}$ (which we recall is a hypothesis from Theorem \ref{linthm}) if we take $C_1=C_1(A)\ge C_0$), we have that 
$$K_{k, v}^A\subset K_{r, s, v}^A$$
for any $r$.
We will be interested in estimating
\begin{equation}\label{e:C(s,N)}
C(s,V) := \left\| \sum_{2^{s-1} \leq r < 2^s, \ v \in V} \mathbf{1}_{K_{r,s,v}^A} \right\|_{L^\infty(\mathbb{T}^2\setminus B_{A^{-2}}(0))}.
\end{equation}
If we let $A\mathbb{T}^2$ denote the dilation of the torus by a factor of $A$, then by scaling considerations, with
\[ K^{A, *}_{r, s, v} := \left\{ \beta \in A\mathbb{T}^2 :  |\beta| \gtrsim A^{-1}, \  | v \cdot \beta - \frac{b}{r}| \lesssim 2^{-C_1 s} \text{ for some } b \leq Ar, b \in \Z \right\},\]
it suffices to estimate
\begin{equation}\label{e:C(s,N)1} C(s, V)\approx\left\| \sum_{2^{s-1} \leq r < 2^s, \ v \in A^{-1} V} \mathbf{1}_{K^{A, *}_{r, s, v}} \right\|_{L^\infty(\mathbb{AT}^2\setminus B_{A^{-1}}(0))}.
\end{equation}
So now we may work with vectors $v\in A^{-1}V$ that have modulus $\approx 1$, so long as they can be rescaled so that each $AV$ is a lattice point. If each $v$ has modulus $\approx 1$, we may view the set $K_{r, s, v}^{A, *}$ as a union of parallel, equally spaced tubes in $A\mathbb{T}^2$ with long axes perpendicular to $v$, of thickness $\approx 2^{-C_0s}$ and spacing $\approx 1/r$.
Our main result in this section is the following proposition.
\begin{proposition}[Incidence estimate]\label{t:C(s,N)fornow1}
Let $\epsilon> 0$, $N>0$. Then there is $C_0=C_0(\epsilon)>0$ so that for any $A\ge N^{C_0}$ there exist a collection of vectors $V=V_{N, \epsilon}\subset\{x\in\mathbb{Z}^2:\,\frac{1}{100}A\le |x|\le 100A\}$ with $|V|=N$ so that for each $N^{\epsilon} \leq 2^s \leq N^{\frac{1}{\epsilon}}$, one may take
\begin{equation}\label{e:C(s,N)est1}
C(s,V) \lesssim_\epsilon N^{\epsilon}
\end{equation}
whenever the implicit constant $C_1=C_1(A)$ in the definition of $C(s, V)$ is taken sufficiently large.
\end{proposition}

We may relate the quantity $C(s, V)$ to mixed norm estimates for the $X$-ray estimate in the plane. Such estimates are equivalent to certain $L^p$ norm estimates for sums of characteristic functions of tubes, and the endpoint $L^2$ $X$-ray estimate in the plane may be stated as follows.

\begin{proposition}[Endpoint $X$-ray estimate in the plane] 
Let $1 \gg \delta > 0$ be arbitrary, and let $\mathbf{T} := \{T\}$ be a set of $\delta$--tubes in $\RR^2$ so that no tube is contained in the two-fold dilate of any other tube. Suppose furthermore that for each direction $v\in \mathbb{S}^1$, at most $m$ tubes point in a direction that is $\delta$--close to $v$. Then for every $\epsilon>0$,
    \begin{equation}
    \Big\Vert\sum_{T \in\mathbf{T}}\mathbf{1}_T \Big\Vert_{2}\lesssim_\epsilon \delta^{-\epsilon}\cdot m^{1/2} \cdot \left( \delta |\mathbf{T}| \right)^{1/2}.
    \end{equation}
\end{proposition}
This is similar to a dualized Kakeya maximal estimate, except the $X$-ray estimate involves multiple tubes in the same direction.

An estimate for $C(s, V)$ may thus be viewed as an $L^{\infty}$ arithmetic $X$-ray estimate in prescribed directions, since it involves sums of characteristic functions of tubes lying in a restricted set of directions such that centers of parallel tubes lie in arithmetic progression.

We now turn to the proof.

\subsection{The Proof of Proposition \ref{t:C(s,N)fornow1}}

Recall our set-up: for some fixed $A$, we have
\begin{equation}\label{e:Vloc}
V \subset \{ x \in \mathbb{Z}^2 : |x| \approx A\};
\end{equation}
we will work with $V_A := \frac{1}{A} \cdot V$. It suffices to choose a collection of vectors $V_A=\{v_1, v_2, \ldots, v_N\}$ living in the unit annulus in $\mathbb{R}^2$, so that 
$$V=AV_A\subset\{x\in\mathbb{Z}^2:\,|x|\approx A\}$$ with $|V_A|=N$ so that for any subcollection $S\subset V$ with $|S|=N^{\epsilon}$  and any choice of integers $\{r_{v} : 2^{s-1}\le r_{v}<2^s, v\in S \}$, no point of $\mathbb{R}^2$ is contained in 
\[ \bigcap_{r_{v} : v \in S} K^{A, *}_{r_{v}, s, v}.\]
Here we have defined
\[ K^{A, *}_{r, s, v} := \left\{ \beta \in A\mathbb{T}^2 :  |\beta| \gtrsim A^{-1}, \  | v \cdot \beta - \frac{b}{r}| \lesssim 2^{-C_1 s} \text{ for some } b \leq Ar, b \in \Z \right\},\]
to be a union of parallel tubes in $A\mathbb{T}^2$ with long axes perpendicular to $v$ and with thickness $\approx 2^{-C_1s}$ and spacing $r^{-1}$. Recall that $C_1=C_1(A, \epsilon)$ is a (sufficiently large) constant, and so the thickness of the tubes is very small comparable to their spacing.

\subsubsection*{The Construction}
Choose a large integer $M\gg 1$, to be determined later, and set $\mathbb{P}_N$ to be the set of primes which do not divide $N$, and collect the smallest $N^{\epsilon /2}$ primes in 
\[ \mathbb{P}_N \cap [N^{M/\epsilon}, 10N^{M/\epsilon}]; \]
call this set $P_M$. By Stirling's approximation
it is possible to define an integer 
\begin{equation}\label{e:kappa}
\kappa \approx_{\epsilon} \epsilon^{-1}
\end{equation}
so that 
\begin{align*}
{ N^{\epsilon/2} \choose \kappa}\approx N.
\end{align*}
We would like to construct our collection of vectors so that when we rescale by $A$ the coordinates of the vectors are integers. Now if $A$ is sufficiently large, say $A\ge N^{2M\kappa/\epsilon}$, and is moreover an integer multiple of $N^{M\kappa/\epsilon}$ (by adjusting our choice of $\kappa$ we may assume that $N^{M\kappa/\epsilon}$ is an integer), then this is indeed satisfied when we choose our set $V_A=\{v_1, \ldots, v_N\}$ so that 
$$(v_i)_x=m_i\cdot Q_i\cdot N^{-m\kappa/\epsilon}\cdot p_{i_1}p_{i_2}\cdots p_{i_{\kappa}}$$
$$(v_y)_x=n_i\cdot Q_i\cdot N^{-m\kappa/\epsilon}\cdot p_{i_1}p_{i_2}\cdots p_{i_{\kappa}}$$
and $V_A$ satisfies the following constraints.
\begin{itemize}
\item For each $i$, we have $1/4\le n_i/m_i\le 1/2$ with $n_i, m_i>0$,
\item For each $i$, $(m_i, n_i)\in \mathbb{Z}^2\cap\{x:\,|x|\approx 100N^2\}$, 
\item For all $i\ne j$, we have $(m_i, n_i)^{\perp}\cdot (m_j, n_j)\ne 0$.
\item $\{p_{i_j}:\,1\le j\le\kappa\}\subset P_M$ are distinct,
\item No two $v_i$ have the exact same collection of corresponding primes $\{p_{i_1}, p_{i_2}, \ldots, p_{i_{\kappa}}\}$.
\item Each $Q_i$ is a dyadic number with $N^{-2}\cdot 2^{-100\kappa}\le Q_i\le N^{-2}\cdot 2^{100\kappa}$ chosen so that $|v_i|\approx 1$ (note that there are $\lesssim 100\kappa$ many possible choices of $Q_i$). 
\end{itemize}
The first constraint ensures (for simplicity) that all vectors live in a single quadrant. The second constraint ensures there is a sufficiently high ($\gg N$) multiplicity of directions given by $(m_i, n_i)/|(m_i, n_i)|$ available. The third constraint ensures that the vectors all have distinct directions. The last constraint ensures that $|v_i|\approx 1$. We note that our previous assumption $A\ge N^{2M\kappa/\epsilon}$ shows that we may rescale the collection of vectors $V_A$ by some integer $L_0\in\{x:\,1/50A\le |x|\le 50A\}$ so that $L_0V_A$ lives in the integer lattice $\mathbb{Z}^2$. Indeed, if we set $L_1$ to be the product of all the distinct $Q_i$ that satisfy $Q_i\le 1$, we may take $L_0$ to be anything of the form
$$L_0=jN^{m\kappa/\epsilon}L_1,\qquad j\in\mathbb{N}$$
chosen so that $L_0\in\{x:\,1/50A\le |x|\le 50A\}$, and this is possible if $A\ge N^{2M\kappa/\epsilon}\gg N^{m\kappa/\epsilon}\prod_{i:\,Q_i\le 1}Q_i^{-1}$.

\subsubsection*{A Geometric Reduction}
Note that by our construction, the difference in angle between any two distinct $v_i$'s is at least $\gtrsim 1/N^2$, so, taking $C_1$ sufficiently large $(\gg 1$), the intersection of any two combs $K_{r_1, v_1} \cap K_{r  _2, v_2}$ belongs to a small (say, $N^{-C_1/2}$) neighborhood of a grid, where the $x$-coordinates of all points in the grid belong to a lattice with generators parallel to $v_1^{\perp}$ and $v_2^{\perp}$. One can then compute that this lattice is given by
$$
\frac{\|v_1\|\|v_2\|}{|\left<v_1^{\perp}, v_2\right>|}\cdot\bigg\{\frac{a}{r_1}(v_1^{\perp})+\frac{b}{r_2}(v_2^{\perp})|\, a, b\in\mathbb{Z}, 0\lesssim a, b \lesssim A^2\bigg\}.
$$
Projecting onto the $x$-axis, we see that the $x$-coordinates of all points in this grid belong to an $O(N^{-C_1/2})$-nieghborhood of the rank $2$ arithmetic progression 
\begin{align}\label{e:SET}
\frac{\|v_1\|\|v_2\|}{|\left<v_1^{\perp}, v_2\right>|}\cdot\bigg\{\frac{a}{r_1}(v_1^{\perp})_x+\frac{b}{r_2}(v_2^{\perp})_x|\, a, b\in\mathbb{Z}, 0\lesssim a, b \lesssim A^2\bigg\}.
\end{align}
We thus obtain that $(K_{r_1, v_1}\cap K_{r_2, v_2})_x$ is contained in an $O(N^{-C_1/2})$ neighborhood of {\eqref{e:SET}}, and we may apply a similar argument for $(K_{r_1, v_1}\cap K_{r_2, v_2})_y$.

\subsubsection*{$N^{\epsilon}$-subcollections of $V$ have subsets that exhibit a certain degree of arithmetic independence}
Suppose now that $S \subset V$ has $|S| = N^{\epsilon}$, and choose an integer $K$ so that
\begin{align}
{K \choose \kappa}\approx N^{\epsilon/2}.
\end{align}
Then by Stirling's approximation,
\begin{align}\label{e:K}
K \approx_{\epsilon} N^{\frac{\epsilon^2}{2}}.
\end{align}

The following claim then follows from a counting argument.
\begin{claim}
For any choice of $S\subset V_A$ with $|S|=N^\epsilon$, one may choose at least $K$ different primes $\{ p_{l_1}, p_{l_2}, \ldots, p_{l_K} \} \subset P_M$ so that we can choose disjoint sub-collections $\{ S_1, \dots, S_K \}$ of $S$, each of cardinality $2$, so that if $S_j = \{ v_{i_j}, v_{i_j'}\}$, both
\begin{equation}\label{e:pdiv}
(v_{i_j})_y \cdot N^{M\kappa/\epsilon}\cdot Q_{i_j}^{-1}\text{ and }  (v_{i_j'})_y \cdot N^{M\kappa/\epsilon}\cdot Q_{i_j'}^{-1}
\end{equation}
are divisible by $p_{l_j}$.
\end{claim}
The proof of this claim is essentially a counting argument. In particular, we inductively choose $S_j$ as follows. If $j\leq k$, take $S_j$ to be any two elements not contained in $\bigsqcup_{j'<j}S_{j'}$ such that the two elements are divisible by some common $p_{l_j}\notin\{p_{l_1}, \ldots, p_{l_{j-1}}\}$. If $j\le k$, it is easy to see that this is possible, since the number of elements $v$ of $V_A$ (and hence $S$) satisfying that $(v)_y$ is divisible only by $p_{l_1}, \ldots, p_{l_{j-1}}$ is $\ll N^{\epsilon}$. Thus $\gtrsim N^{\epsilon} $ elements of $S$ remain that are divisible by some prime $p_l$ with $l\ne l_1, \ldots, l_{j-1}$, and since there are $ N^{\epsilon/2} $ many primes to choose from and $\approx  N^{\epsilon}$ different elements of $S$ that have not yet been chosen, by pigeonholing there must be some $p_l=: p_{l_j}$ for which at least two remaining elements of $S$ are divisible by $p_l$. This completes the proof of the claim.
\newline
\newline
Now suppose we are given some $S\subset V_A$ with $|S|=N^{\epsilon}$. Then by the above claim, if $C_1\gg 100M\kappa/\epsilon^3$, we have
$$
 \frac{|\left<v_{i_j}^{\perp}, v_{i_j'}\right>|}{\|v_{i_j}\|\|w_{i_j}\|}\cdot N^{M\kappa/\epsilon} \cdot r_{i_j}r_{i_j'}\cdot\bigg(K_{r_{i_j}, v_{i_j}}\cap K_{r_{i_j'}, v_{i_j'}}\bigg)_x\subset p_{l_j}\mathbb{Z}+O(N^{-C_1/4}).
$$
Note that the scaling factor on the left hand side is chosen to make the rescaled $x$-projection of the lattice corresponding to  to $K_{r_{i_j}, v_{i_j}}\cap K_{r_{i_j'}, v_{i_j'}}$ be an integer. Now taking $N^{C_1}\gg A^{100}$, we may square both sides and obtain
$$
 \frac{|\left<v_{i_j}^{\perp}, v_{i_j'}\right>|^2}{\|v_{i_j}\|^2\|w_{i_j}\|^2}\cdot N^{2M\kappa/\epsilon} \cdot r_{i_j}^2r_{i_j'}^2\cdot\bigg(K_{r_{i_j}, v_{i_j}}\cap K_{r_{i_j'}, v_{i_j'}}\bigg)^2_x\subset p_{l_j}\mathbb{Z}+O(N^{-C_1/5}).
$$
Note that 
$$
 \frac{|\left<v_{i_j}^{\perp}, v_{i_j'}\right>|^2}{\|v_{i_j}\|^2\|w_{i_j}\|^2}=\frac{|m_{i_j}n_{i_j'}-n_{i_j}m_{i_j'}|}{(m_{i_j}^2+n_{i_j}^2)(m_{i_j'}^2+n_{i_j'}^2)}
 $$
 with $m_{i_j}, n_{i_j}, m_{i_j'}, n_{i_j'}\lesssim N^2$, we have
 $$
|m_{i_j}n_{i_j'}-n_{i_j}m_{i_j'}|\cdot N^{2M\kappa/\epsilon} \cdot r_{i_j}^2r_{i_j'}^2\cdot\bigg(K_{r_{i_j}, v_{i_j}}\cap K_{r_{i_j'}, v_{i_j'}}\bigg)^2_x\subset p_{l_j}\mathbb{Z}+O(N^{-C_1/6}).
$$
But this implies that $K'\le C_1/1000$, then we have
\begin{multline*}
\bigcap_{j=1}^{K'}\bigg(K_{r_{i_j}, v_{i_j}}\cap K_{r_{i_j'}, v_{i_j'}}\bigg)^2_x
\\
\subset \bigg[N^{2M\kappa/\epsilon}\prod_{j=1}^{K'}\bigg(|m_{i_j}n_{i_j'}-n_{i_j}m_{i_j'}|\cdot r_{i_j}^2r_{i_j'}^2\bigg)\bigg]^{-1}\cdot \bigg(\prod_{j=1}^{K'}p_{l_j}\bigg)\mathbb{Z}+O(N^{-C_1/7}).
$$
\end{multline*}
(Here the notation $(S)^2$ on the left hand side of the previous inequality  where $S=K_{r_{i_j}, v_{i_j}}\cap K_{r_{i_j'}, v_{i_j'}}$ indicates the pointwise squared elements of the set $S$, that is the set $\{x^2:\,x\in S\}$.) For $K'=C_1/1000$, we estimate
$$\bigg[N^{2M\kappa/\epsilon}\prod_{j=1}^{K'}\bigg(|m_{i_j}n_{i_j'}-n_{i_j}m_{i_j'}|\cdot r_{i_j}^2r_{i_j'}^2\bigg)\bigg]\lesssim N^{(2M/\epsilon^3+C_1/100)/\epsilon}.$$
But
$$\prod_{j=1}^Kp_{l_j}\gtrsim N^{C_1M/(1000\epsilon)}.$$
For $M$ sufficiently large $(i.e. M\gg 1000\epsilon$) we thus have 
$$
\bigg|\bigg[N^{2M\kappa/\epsilon}\prod_{j=1}^{K'}\bigg(|m_{i_j}n_{i_j'}-n_{i_j}m_{i_j'}|\cdot r_{i_j}^2r_{i_j'}^2\bigg)\bigg]^{-1}\cdot \bigg(\prod_{j=1}^{K'}p_{l_j}\bigg)(\mathbb{Z}\setminus\{0\}\bigg|\gg N^{C_1/100}.$$
But $N^{C_1/100}\gg A^2\approx N^{20M/\epsilon^3}$ if $C_1$ is sufficiently large $(C_1\gg 2000 M/\epsilon^3$). Since
$$
\bigcap_{j=1}^{K'}\bigg(K_{r_{i_j}, v_{i_j}}\cap K_{r_{i_j'}, v_{i_j'}}\bigg)^2_x\subset [-1000A^2, 1000A^2],
$$
it follows that
$$
\bigcap_{j=1}^{K'}\bigg(K_{r_{i_j}, v_{i_j}}\cap K_{r_{i_j'}, v_{i_j'}}\bigg)^2_x\subset [-N^{-C_1/7}, N^{-C_1/7}].
$$
By a symmetric argument, we may also show that 
$$
\bigcap_{j=1}^{K'}\bigg(K_{r_{i_j}, v_{i_j}}\cap K_{r_{i_j'}, v_{i_j'}}\bigg)^2_y\subset [-N^{-C_1/7}, N^{-C_1/7}],
$$
and hence
$$\bigcap_{j=1}^{K'}\bigg(K_{r_{i_j}, v_{i_j}}\cap K_{r_{i_j'}, v_{i_j'}}\bigg)\subset B_0(N^{-C_1/14}).$$
But for $C_1$ sufficiently large, we have $B_0(N^{-C_1/14})\subset B_{A^{-1}(0)}$, and by definition any comb $K_{r_{i_j}, v_{i_j}}$ has empty intersection with $B_{A^{-1}}(0)$, and so this completes the proof.

\section{The Proof of Theorem \ref{linthm} and Corollary \ref{cor:multiscalemax}}\label{s:easyproof}
We now note that Theorem \ref{linthm} follows by the previous estimate for $C(s, N)$. Indeed, as previously alluded to, we show that an upper bound for
\begin{equation}\label{e:opNorm}
\| A_{V,k} \|_{\ell^2(\mathbb{Z}^2) \to \ell^2(\mathbb{Z}^2)}
\end{equation}
 is given by the sum of the following two terms:
\begin{equation}\label{e:term2}
\| \sup_{v \in V} |A_{V,k} (\varphi*f)(x)| \|_{\ell^2 \to \ell^2}
\end{equation}
and
\begin{equation}\label{e:term1}
\sup_{ |\beta| \gtrsim A^{-2} } \left( \sum_{v \in V} \sum_{0 < m \lesssim A} \mathbf{1}_{|v \cdot \beta - m | \ll 2^{-k}}(\beta) \right)^{1/2}.
\end{equation}
In \eqref{e:term2}, $\varphi$ is an appropriate Schwartz function with Fourier transform supported in $\{ |\xi| \lesssim A^{-1} 2^{-k} \}$. Clearly, we have
\begin{multline}
\| A_{V,k} \|_{\ell^2(\mathbb{Z}^2) \to \ell^2(\mathbb{Z}^2)}
\\
\lesssim\| \sup_{v \in V} |A_{V,k} (\varphi*f)(x)| \|_{\ell^2 \to \ell^2}+\| \sup_{v \in V} |A_{V,k} (\varphi\ast f-f)(x)| \|_{\ell^2 \to \ell^2}
\end{multline}
So it suffices to show that
\begin{equation}
\| \sup_{v \in V} |A_{V,k} (\varphi\ast f-f)(x)| \|_{\ell^2 \to \ell^2}^2\lesssim \sup_{ |\beta| \gtrsim A^{-2}}  \sum_{v \in V} \sum_{0 \leq m \lesssim A} \mathbf{1}_{|v \cdot \beta - m | \ll 2^{-k}}(\beta).
\end{equation}
Dominating the $\sup$ in $v$ by a square sum, we have
\begin{equation}\label{t1}
\| \sup_{v \in V} |A_{V,k} (\varphi\ast f-f)(x)| \|_{\ell^2}^2\lesssim\sum_{v\in V}\bigg\|A_{v, k}(\varphi*f-f)\bigg\|_{\ell^2}^2,
\end{equation}
where $A_{v, k}f=\sum_{n\in\mathbb{Z}^2}f(x-nv)\phi_k(n)$.
Now note that the operator $A_{v, k}$ is a multiplier operator given by
\begin{multline}
A_{V, k}f=\sum_{n\in\mathbb{Z}}f(x-nv)\phi_k(n)=\int_{\mathbb{T}^2}\sum_{n\in\mathbb{Z}}\phi_k(n)\widehat{f}(\beta)e((x-nv)\cdot\beta)
\\
=\int_{\mathbb{T}^2}\widehat{f}(\beta)e(x\cdot\beta)\bigg(\sum_{n\in\mathbb{Z}}\phi_k(n)e(-nv\cdot\beta)\bigg)\,d\beta=\int_{\mathbb{T}^2}\widehat{f}(\beta)e(x\cdot\beta)\widehat{\phi_k}(-v\cdot\beta)\,d\beta,
\end{multline}
and hence $A_{V, k}$ is a Fourier multiplier operator with multiplier $\widehat{\phi_k}(-v\cdot\beta)$. Composing $A_{V, k}$ with the operator with multiplier $(1-\phi)$ is hence dominated pointwise by the restriction to $\mathbb{T}^2$ by $1_{K^A_{k, v}}$, where here $K^A_{k, v}$ (off of the ball centered at the origin of radius $\approx A^{-2}$) is the characteristic function of a union of parallel tubes of thickness $\approx 2^{-k}A^{-1}$ spaced apart by $\approx A^{-1}$, with long axes perpendicular to $v$. That is, we recall that we have defined
$$
K^A_{k, v}:=\{\beta\in\mathbb{T}^2, |\beta|\gtrsim A^{-2}:\,|v\cdot\beta-m|\lesssim 2^{-k}A^{-1}\,\text{for some }m\in\mathbb{Z}\}.
$$
Clearly, $A_{v, k}$ is bounded on $\ell^2(\mathbb{Z}^2)$, and hence by Plancherel the right hand side of (\ref{t1}) is bounded by
$$
\sum_{v\in V}\|1_{K_{A, k, v}}\widehat{f}\|_{\ell^2}^2\lesssim\bigg\|\bigg(\sum_{v\in V}|1_{K_{A, k, v}}\widehat{f}|^2\bigg)^{1/2}\bigg\|_{\ell^2}^2
\lesssim\bigg\|\bigg(\sum_{v\in V}1_{K_{A, k, v}}|^2\bigg)^{1/2}\bigg\|_{\ell^{\infty}}\|f\|_{\ell^2},$$
and hence the $\ell^2(\mathbb{Z}^d)\to\ell^2(\mathbb{Z}^d)$ operator norm is given by (\ref{e:term1}). Note that for the choice of $s$ such that $2^s=N$ and for $2^k\ge N^{C_0}$ (which we recall is a hypothesis from Theorem \ref{linthm}), we have that 
$$K_{k, v}^A\subset K_{r, s, v}^A$$
for any $r$. Thus by our incidence estimate with $2^s=N$, this term is $\lesssim_{\epsilon}N^{\epsilon}$ when $V=V_{N, \epsilon}$ and $2^k\ge N^{C_0}$.
We begin with \eqref{e:term2}:

Consider the following decomposition:
\[ A_{V,k} (\varphi*f) = A_{V,k}^{\mathbb{R}} f + \mathcal{E}_v f,\]
where $A_{V,k}^{\mathbb{R}}$ is a convolution operator with multiplier given by
\[ m_v(\beta) := m_{v,k}(\beta) := \hat{\varphi}(\beta) \cdot \sum_{m \in \mathbb{Z}} \widehat{ \phi_k }( v \cdot \beta - m), \]
and $\{ \mathcal{E}_v :v \}$ are error terms, with uniformly small Fourier coefficients:
\[ \sup_{v, \beta} |\widehat{\mathcal{E}_v}(\beta)| \lesssim 2^{-k}.\]
(Here, we are conflating the convolution operator and its multiplier.)
In particular, we may estimate
\[\|  \sup_{v} |\mathcal{E}_v f| \|_{\ell^2} \lesssim N^{1/2} \cdot 2^{-k} \cdot \|f\|_{\ell^2}.\]

To handle the contribution of $A_{V,k}^{\mathbb{R}}$ we will need the following special case of a beautiful transference argument of Magyar, Stein, and Wainger \cite[Lemma 2.1]{MSW}.

\begin{lemma}\label{trans}
Let $B_1,B_2$ be finite-dimensional Banach spaces, and 
\[ m: \RR^2 \to L(B_1,B_2) \]
be a bounded function supported on a cube with side length one containing the origin that acts as a Fourier multiplier from
\[ L^p(\RR^2,B_1) \to L^p(\RR^2, B_2),\]
for some $1 \leq p \leq \infty.$ Here, $L^p(\RR^2,B):= \{ f: \RR^2 \to B : \| \| f\|_B \|_{L^p(\RR^2)} < \infty\}$.
Define
\[ m_{\text{per}}(\beta) := \sum_{l \in \Z^2} m(\beta - l) \ \text{ for } \beta \in \TT^2.\]
Then the multiplier operator
\[ \| m_{\text{per}} \|_{\ell^p(\Z^2,B_1) \to \ell^p(\Z^2,B_2)} \lesssim \| m \|_{L^p(\RR^2,B_1) \to L^p(\RR^2,B_2)}.\]
The implied constant is independent of $p, B_1,$ and $B_2$.
\end{lemma}

In particular, we will apply Lemma \ref{trans} for
\[ m(\beta) := \{ m_v(\beta) : v \in V \} \]
which takes values in 
\[ \ell^2(\mathbb{Z}^2;\ell^{\infty}(V)) := \{ f = \{ f_v : v \in V\} : \| \sup_{v \in V} |f_v| \|_{\ell^2(\mathbb{Z}^2)} < \infty \};\]
this accrues a norm loss of $\lesssim_\epsilon N^{\epsilon}$ by the continuous theory.

As for \eqref{e:term1}, we may bound this above by $N^{\epsilon}$ using the above construction (the argument slightly simplifies as no rational shifts are needed); note the necessity of the condition that $k \gg \log N$ (this has the effect of ensure that the quantity $C_0$ used in \S \ref{s:inc} above is sufficiently large).

\bigskip

To extend this result to Corollary \ref{cor:multiscalemax}, we simply observe that
\[ \left( \sum_{v} \sup_{j \geq k} \left| \sum_n f(x-n\cdot v) \cdot \phi_j(n) \right|^2 \right)^{1/2} \]
has $\ell^2$ norm dominated by 
\[ \left( \sum_{v} \left| \sum_n f(x-n\cdot v) \cdot  \tilde{\phi}_k(n) \right|^2 \right)^{1/2} \]
for an appropriate choice of $\phi, \tilde{\phi}$; this follows from the uniform-in-$v\in\mathbb{Z}^2$ $\ell^2$ boundedness of the maximal function
\[ \sup_N \left| \frac{1}{N}  \sum_{n\leq N} f(x-n\cdot v) \right|. \]
This boundedness may be seen by first establishing the weak-type $1-1$ estimate via the Hopf-Dunford-Schwartz maximal inequality from ergodic theory; the method of $TT^*$ also suffices.

\section{The Proof of Theorems \ref{t:2avg} and \ref{t:avg}}\label{s:squares}
The goal of this section is to prove Theorem \ref{t:2avg}, reproduced below for the reader's convenience; Theorem \ref{t:avg} will follow by only minor modifications to the method.
\begin{theorem}
Let $\epsilon > 0$ be arbitrary. Then, for any $N$, there exists a set of vectors $|V| = |V_{N,\epsilon}| = N$ so that, if $k \gg_\epsilon \log N$,
\begin{equation}\label{e:maxpbound}
\| A^{2}_{V,k} f \|_{\ell^2} \lesssim_\epsilon N^{\epsilon} \cdot \| f \|_{\ell^2}.
\end{equation}
\end{theorem}

Our analysis combines discrete-harmonic analytic techniques, the incidence estimates from \S \ref{s:inc}, and the estimate for the linear (non-Radon) maximal function, \eqref{e:Aphi}.

We turn to the details. The initial reductions we make are number-theoretic.

\subsection{Approximations}\label{ss:app}
In what follows, we will assume that $k \gg_\epsilon \log N$. 

Set 
\begin{equation}\label{e:onedker}
K_k:= \sum_{n \in \mathbb{Z}} \phi_k(n) \cdot  \delta_{n^2}, 
\end{equation}
where $\delta_m$ is the point-mass at the point $m$; with this in mind, we may view
\[ A^{2}_{V,k} f \]
as a maximal multiplier operator, where the multipliers are given by 
\[ \{ \widehat{K_k}(v \cdot \beta) : v \}. \]

In particular, it is the (one-dimensional) Fourier transform of $K_k$ that we need to consider.

We will construct our analytic family of approximating multipliers by analyzing the behavior of $\widehat{K_k}$ on the \emph{major arcs}, which we now proceed to define. A reference for the material appearing below is \cite[\S 5-6]{B1}.

\subsubsection{Major Arcs}
Throughout this section, $\epsilon_0,\epsilon_1 \ll_\epsilon 1$ are sufficiently small constants (as one sends $\epsilon_0,\epsilon_1$ down to zero, the size of $k$ relative to $\log N$ will increase.)

We begin by collecting the rationals in the torus according to the size of their denominators; roughly speaking, these sets form approximate level sets for appropriate complete exponential sums (below).

For each $s \geq 1$, set
\begin{equation}\label{e:Rs}
\R_s := \{ a/q  \in \TT \text{ reduced} : 2^{s-1 }\leq q < 2^{s} \}.
\end{equation}

\begin{definition}  \label{d:major}
For $a/q \in  \mathcal R_s$, where $ s\leq \epsilon_0 k $, define the \emph{$ k$th major arc at $a/q$} to be given by 
\begin{equation}\label{e:major}
\mathfrak{M}_k(a/q) := \left\{ \beta \in \mathbb{T} : | \alpha - a/q| \leq 2^{(\epsilon_1 - 2)k} \right\}.
\end{equation}
\end{definition}

On each $\mathfrak{M}_k(a/q)$, we may extremely accurately approximate the multiplier $\widehat{K_k}(\alpha)$. To do so, we introduce the complete exponential sums,
\begin{equation}\label{WEYL}
S(a/q) := \frac{1}{q} \sum_{r \leq q} e(- a/q \cdot r^2), \ (a,q) = 1. 
\end{equation}
By ``completing the square'' (and using periodicity of the phases $r \mapsto e(- a/q \cdot r^2)$, it is straightforward to check that
\begin{equation}\label{e:sqrtcan}
|S(a/q)| \lesssim q^{-1/2}.
\end{equation}
If one wishes to replace the monomial $r \mapsto r^2$ with a more general integer-valued polynomial $P$ of degree $d$, the estimate we use is due to Hua, \cite[\S 7, Theorem 10.1]{Hua}, where the savings are on the order of $q^{\epsilon - 1/d}$. For our purposes here, the key point is that, for each $d$, there exists some $\delta_d > 0$ so that $|S(a/q)| \lesssim q^{-\delta_d}$.

We will also need the continuous analogue of our discrete Fourier transform,
\begin{equation}\label{e:V_k1}
V_k(\xi) := \int e(-2^{2k}t^2 \xi) \phi(t) \ dt.
\end{equation}

Our main ``local'' characterization of $\widehat{K_k}(\alpha)$ is contained in the following proposition.
\begin{proposition}\label{p:onMAJ}
On $\mathfrak{M}_k(a/q)$, we may express
\begin{equation}\label{onMAJ}
\widehat{K_k}(\alpha) = S(a/q) V_k(\alpha - a/q) + O(2^{(2\epsilon_1 - 1)k}).
\end{equation}
\end{proposition}
The proof of this proposition is entirely elementary.
\begin{proof}
By a summation by parts argument, taking into account the smoothness of the function $\phi$, it suffices to assume that 
\[ \phi_k := \frac{1}{2^k} \mathbf{1}_{[1,2^k]}.\]

Now, set $\alpha = a/q + \eta$, where $|\eta| \lesssim 2^{(\epsilon_1 -2)k}$. Working $\mod 1$, and expressing $r = lq + p$, we have
\[ \alpha \cdot r^2 = (a/q + \eta) \cdot (lq + p)^2 \equiv a/q \cdot p^2 + \eta \cdot (lq)^2 + O(2^{(2\epsilon_1 - 1)k}).\]
Consequently, we may express
\[ \aligned 
\widehat{K_k}(\alpha) &= \frac{1}{2^k} \sum_{n \leq 2^k} e(-\alpha n^2) \\
&= \sum_{ p \leq q} e(-a/q \cdot p^2) \sum_{l \leq 2^k/q} e(-\eta \cdot (lq)^2) + O( 2^{(2\epsilon_1 -1)k}), \endaligned\] 
where we have inserted and removed some terms, all absorbed into the big-$O$ notation.
Since the phase $l \mapsto \eta q^2 \cdot l^2$ has a tiny derivative (on the order of $2^{(2\epsilon_1 - 1)k}$), we may use a Riemann-sum approximation to replace
\[ \frac{q}{2^k} \sum_{l \leq 2^k/q} e(-\eta \cdot (lq)^2) = V_k(\eta) + O(2^{(2\epsilon_1 - 1)k}),\]
from which the result follows.
\end{proof}

In light of this proposition, we know how to locally approximate $\widehat{K_k}$. To approximate it \emph{globally}, we need to ``patch together'' our local approximates. Doing so takes a little more notation.

First, we collect the major boxes 
\begin{equation}\label{e:Major}
 \mathfrak{M}_k := \bigcup_{s \leq \epsilon_0 k} \bigcup_{a/q \in \mathcal{R}_s}  \mathfrak{M}_k(a/q). 
\end{equation}

For $s \leq \epsilon_0 k$, we define the multipliers
\begin{equation}\label{L_{k,s}}
L_{k,s}(\alpha) := \sum_{a/q \in \mathcal{R}_s} S(a/q) V_k(\alpha - a/q) \chi_k(\alpha - a/q)
\end{equation}
where $\chi_k(\alpha) := \chi(2^{(2-\epsilon_1)k} \alpha)$, where $\epsilon_1$ is as in the definition of the major arcs, \eqref{e:major}. Note that the sum over $a/q \in \mathcal{R}_s$ in the definition of \eqref{L_{k,s}} is over disjointly supported terms.

We define
\begin{equation}\label{L_k}
L_k := \sum_{ s \leq \epsilon_0 k} L_{k,s}.
\end{equation}

Then, we have the following important proposition.

\begin{proposition}\label{MAINNT} There exists some $\kappa > 0$ so that
\[ \| \widehat{K_k} - L_k \|_{L^\infty(\mathbb{T})} \lesssim 2^{-\kappa k}.\]
\end{proposition}

Before turning to the proof of this theorem we first recall \emph{Weyl's Lemma}.

\begin{lemma}\label{WEYLL}[Weyl's Lemma]
Suppose $|a_d - a/q| \leq \frac{1}{q^2}$. Then
\[ \left| \sum_{n \leq N} e( a_d \cdot n^d + \dots a_1 n) \right| \lesssim_\epsilon N^{1+\epsilon}\left(\frac{1}{N} + \frac{1}{q} + \frac{q}{N^d} \right)^{\frac{1}{2^{d-1}}}.\]
\end{lemma}
\begin{remark}
The power $\frac{1}{2^{d-1}}$ is classical, but not sharp; see \cite{W} for an improvement.
\end{remark}

With this in mind, we turn to the proof of Theorem \ref{MAINNT}.

\begin{proof}[Proof of Theorem \ref{MAINNT}]
We will establish this in cases. First we assume that $a/q \in \mathcal{R}_{s_0}$ for some $s_0 \leq \epsilon_0 k$, and consider the behavior of $\widehat{K_k} - L_k$ on $\mathfrak{M}_k(a/q)$.

First observe that 
\[ \widehat{K_k} - L_{k,s_0} \]
vanishes
on $\mathfrak{M}_k(a/q)$. 
Suppose now that $s \neq s_0 \leq \epsilon k$. Then
\begin{equation}\label{dec}
\sup_{\alpha \in \mathfrak{M}_k(a/q) } |L_{k,s}(\beta)| \lesssim 2^{-s/2} \min_{b/r \in \mathcal{R}_s} |V_k(\alpha - b/r)| \lesssim 2^{-s/2} \cdot 2^{(\epsilon_0 - 1)k},
\end{equation}
which sums nicely over $s$; note our use of the lower bound
\[ \aligned 
|\alpha - b/r| &\geq |a/q - b/r| - 2^{(\epsilon_1-2)k} \\
& \qquad \geq \frac{1}{qr} - 2^{(\epsilon_1-2)k} \gtrsim 2^{-2 \epsilon_0 k}, \endaligned \]
for $\alpha \in \mathfrak{M}_k(a/q)$.

On the minor arcs, $\alpha \notin \mathfrak{M}_k$, a similar calculation to \eqref{dec} shows that $L_k$ has a power savings; as for $\widehat{K_k}$, by Dirichlet's principle one may choose a reduced rational, $b/r$, with $r \lesssim 2^{(2-\epsilon_1)k}$, so that
\[ | \alpha - b/r| \lesssim \frac{1}{r 2^{(2-\epsilon_1)k}} \leq \frac{1}{r^2}.\]
We are done unless $r \lesssim 2^{\epsilon_0 k}$ by Weyl's Lemma \ref{WEYLL}. But in this case, we have shown that $\alpha \in \mathfrak{M}_k(b/r)$, which is a contradiction.
\end{proof}

The upshot from this discussion is that we may pass from the maximal function
\[ \sup_{v} | \left( \widehat{K_k}(v \cdot ) \hat{f} \right)^{\vee} | \]
to the analytic approximate
\begin{equation}\label{e:analmax}
\sup_{v} | \left( \widehat{L_k}(v \cdot ) \hat{f} \right)^{\vee} |,
\end{equation}
via a square function argument, since
\[ \left( \sum_{v} | \left( \big(L_k - \widehat{K_k}\big)(v \cdot ) \hat{f} \right)^{\vee} |^2 \right)^{1/2}\]
is bounded on $\ell^2$ for $k \gg \log N$ sufficiently large.

In the next subsection, we will turn our attention to the maximal function, \eqref{e:analmax}.
\subsection{Estimating The Maximal function}
In light of the previous section, henceforth we need only consider the sum in $1 \leq s \leq \epsilon_0 k$ of the following operators,
\begin{equation}\label{e:maxPrimes-s}
\mathcal{L}_{V,k,s} f:= \sup_{v} \left| \left(L_{k,s}( v \cdot \beta) \hat{f}(\beta) \right)^{\vee} \right|
\end{equation}

With this in mind, we are ready for the proof of Theorem \ref{t:2avg}.

\subsubsection{The Proof of Theorem \ref{t:2avg}}
It suffices to prove the following proposition.
\begin{proposition}\label{p:2est}
Let $\epsilon > 0$ be arbitrary. Then there exists a set $|V| = N$ so that for each $1 \leq s \leq \epsilon_0 k$, with $2^{\epsilon s} \leq N \leq 2^{s/\epsilon}$, we may bound
\begin{equation}\label{e:Lprimesgoal}
\| \mathcal{L}_{V,k,s} f \|_2 \lesssim 
N^{\epsilon} \cdot \|f \|_2.
\end{equation}
\end{proposition}
Indeed, with this proposition in hand, we can quickly deduce Theorem \ref{t:2avg}.
\begin{proof}[Proof of Theorem \ref{t:2avg} Assuming Proposition \ref{p:2est}]
For $s \ll \log N$, we use the triangle inequality to simply bound,
\[ |L_{k,s} f| \lesssim 2^{-s/2} \cdot \sum_{a/q \in \mathcal{R}_s} \sup_v \left| \sum_m \left( V_k(v \cdot \beta - a/q -m ) \chi_k(v \cdot \beta - a/q -m) \hat{f}(\beta) \right)^{\vee} \right| \]
and apply the linear theory, Theorem \ref{linthm}, to control 
\[ \| \sup_v \left| \sum_m \left( V_k(v \cdot \beta - a/q -m) \chi_k(v \cdot \beta - a/q-m) \hat{f}(\beta) \right)^{\vee} \right| \|_{\ell^2} \lesssim N^{\epsilon} \| f\|_{\ell^2}\]
by convexity. In particular, an upper bound for the contribution of $s \ll \log N$ is 
\[ N^{C \cdot \epsilon}  \cdot \|f \|_2.\]
For $s \gg \log N$, we replace the supremum in $v \in V$ with a square-sum,
\[ |L_{k,s} f|^2 \leq \sum_{v} \left| \sum_m \sum_{a/q \in \mathcal{R}_s} S(a/q) \left( V_k(v \cdot \beta - a/q-m) \chi_s(v \cdot \beta - a/q-m) \hat{f}(\beta) \right)^{\vee} \right|^2,\]
and then take into account the geometric decay of the Weyl sums,
\[ \max_{a/q \in \mathcal{R}_s} |S(a/q)| \lesssim 2^{-s/2}.\]
\end{proof}

It remains to prove Proposition \ref{p:2est}
\begin{proof}[The Proof of Proposition \ref{p:2est}]
For each $2^{s-1} \leq q < 2^s$, let
\[ \widehat{f_q} := \hat{f} \cdot \mathbf{1}_{ \bigcup_v (K^A_{q, s, v})' }, \]
where $K^A_{q, s, v}$ is defined in \eqref{e:Kr}, and $(K^A_{q, s, v})'$ is defined similarly, but only \emph{reduced} rationals $a/q, \ (a,q) = 1$ are considered; this is done so that for each $v \in V$,
\[ \sum_{q= 2^{s-1} }^{2^{s}-1} \mathbf{1}_{(K^A_{q, s, v})'} \]
is bounded by $1$ on $\mathbb{T}^2$.

We now have
\[
\aligned 
& \sup_{v} \left| \left(L_{k,s}( v \cdot \beta) \hat{f}(\beta) \right)^{\vee} \right| \\
& \qquad =
\sup_{v} \left| \left( \sum_m \sum_{2^{s-1} \leq q < 2^s } \sum_{a \leq q, (a,q) = 1} S(a/q) V_k(v \cdot \beta - a/q -m) \chi_k(v\cdot \beta - a/q - m) \widehat{f_q}(\beta) \right)^{\vee} \right|. \endaligned \]

We now simply replace the supremum in $v$ with a square sum; by the same reductions that we used in the proof of Theorem \ref{linthm}, it suffices to turn our attention to
\[ \sup_{\beta\gtrsim A^{-2}} \ 2^{-s/2} \cdot \left( \sum_{v} \sum_q \mathbf{1}_{(K^A_{q, s, v})'} \right)^{1/2}(\beta);\]
but this final expression is bounded by 
\[ 2^{-s/2} \cdot N^{\epsilon} \]
using the incidence estimate from Proposition \ref{t:C(s,N)fornow1}; replacing the sequence $n \mapsto n^2$ sees minor changes in the number theoretic approximation, see \cite[\S 5-6]{B3}, and replacing the gain of $2^{-s/2}$ from the Gauss sums by the corresponding more modest gain, of (say)
\[ 2^{- \frac{s}{2 \cdot \text{deg}(P)} }.\]
The proof is complete.
\end{proof}

\typeout{get arXiv to do 4 passes: Label(s) may have changed. Rerun}


\begin{thebibliography}{9}
\bibitem{Ba}
J. Barrionuevo.
A note on the Kakeya maximal operator.
Math. Res. Lett. 3 (1996), no. 1, 61?65. 

\bibitem{B1}
J.~Bourgain.
\newblock On the maximal ergodic theorem for certain subsets of the integers.
\newblock {\em Israel J. Math.}, 61(1):39--72, 1988.

\bibitem{B2}
J.~Bourgain.
\newblock On the pointwise ergodic theorem on {$L\sp p$} for arithmetic sets.
\newblock {\em Israel J. Math.}, 61(1):73--84, 1988.

\bibitem{B3}
J. Bourgain.
\newblock Pointwise ergodic theorems for arithmetic sets.
\newblock {\em Inst. Hautes \'Etudes Sci. Publ. Math.}, (69):5--45, 1989.
\newblock With an appendix by the author, Harry Furstenberg, Yitzhak Katznelson
  and Donald S. Ornstein.


\bibitem{C}
A. Calder\'{o}n. Ergodic theory and translation invariant operators, Proc. Nat. Acad. Sci., USA
59 (1968), 349-353

\bibitem{Cor}
A. C\'{o}rdoba. The Kakeya maximal function and the spherical summation multipliers, Amer. J. Math. 99 (1977),
no. 1, 1?22


\bibitem{D}
C. Demeter. Singular Integrals Along $N$ Directions in $\mathbb{R}^2$.
Proc. Amer. Math. Soc. 138 (2010), no. 12, 4433?4442. 


\bibitem{DpP}
F. di Plinio and I. Parissis.
Maximal Directional Operators Along Algebraic Varieties.
Preprint, https://arxiv.org/pdf/1807.08255.pdf.


\bibitem{Hua}
L.~K. Hua.
\newblock {\em Introduction to number theory}.
\newblock Springer-Verlag, Berlin-New York, 1982.
\newblock Translated from the Chinese by Peter Shiu.

\bibitem{IW} 
{A. Ionescu, S. Wainger,}
{$L^p$ boundedness of discrete singular Radon transforms.}
{{J. Amer. Math. Soc.} {19}, (2005), no. 2, 357?-383.}

\bibitem{Katz}
N. Katz. Remarks on maximal operators over arbitrary sets of directions, Bull.
London Math. Soc. 31 (1999), 700-710.


\bibitem{Katz1}
N. Katz.
Maximal operators over arbitrary sets of directions, Duke Math. J. 97 (1999),
no. 1, 67?79

\bibitem{KaTao}
N. Katz, T. Tao.
Some connections between Falconer's distance set conjecture and sets of Furstenburg type. New York J. Math. 7 (2001), 149?187.


\bibitem{Kes}
R. Kesler.
$\ell^p(\mathbb{Z}^d)$-Improving Properties and Sparse Bounds for Discrete Spherical Maximal Averages.
Preprint, https://arxiv.org/pdf/1805.09925.pdf

\bibitem{KesLac}
R. Kesler and M. Lacey.
$\ell^p$-Improving and Sparse Inequalities for Discrete Spherical Averages.
Preprint, https://arxiv.org/pdf/1804.09845.pdf.

\bibitem{KLM}
R. Kesler, M. Lacey, D. Mena Arias. 
An Endpoint Sparse Bound for the Discrete Spherical Maximal Functions.
Preprint, https://arxiv.org/pdf/1810.02240.pdf.

\bibitem{BK} 
{B. Krause},
Discrete Analogues in Harmonic Analysis: Maximally Modulated Singular Integrals Related to Carleson's Theorem.
{Preprint, https://arxiv.org/pdf/1803.09431.pdf}

\bibitem{KL}
B. Krause and M. Lacey.
\newblock{A Discrete Quadratic Carleson Theorem on $\ell^2$ with a Restricted Supremum.}
\newblock {\em ArXiv: 1512.06918}, December 2015.




\bibitem{MSW}
A. Magyar, E. Stein, and S. Wainger. 
Discrete analogues in harmonic analysis: spherical averages.
Ann. of Math. (2) 155 (2002), no. 1, 189?208. 

\bibitem{M}
M. Mirek.
$\ell^p(\Bbb{Z})$-boundedness of discrete maximal functions along thin subsets of primes and pointwise ergodic theorems.
Math. Z. 279 (2015), no. 1-2, 27?59. 

\bibitem{MST1} M. Mirek, E. Stein, and B. Trojan.
$L^p(\Z^d)$-estimates for discrete operators of Radon type: Maximal functions and vector-valued estimates
Preprint, http://arxiv.org/pdf/1512.07518.pdf

\bibitem{MST2} M. Mirek, E. Stein, and B. Trojan.
$L^p(\Z^d)$-estimates for discrete operators of Radon type: Variational estimates
Preprint, http://arxiv.org/pdf/1512.07523.pdf

\bibitem{MSZK}
M. Mirek, E. Stein, and P. Zorin-Kranich.
Jump Inequalities for Translation-Invariant Operators of Radon Type on $\mathbb{R}^d$ and $\mathbb{Z}^d$.
Preprint.



\bibitem{Sch}
W. Schlag, A generalization of Bourgain?s circular maximal theorem, J. Amer. Math. Soc. 10
(1997), no. 1, 103?122

\bibitem{S}
E. Stein.
Oscillatory integrals related to Radon-like transforms. 
Proceedings of the Conference in Honor of Jean-Pierre Kahane (Orsay, 1993). 
J. Fourier Anal. Appl. 1995, Special Issue, 535?551.

\bibitem{SW}
E. Stein and S. Wainger.
\newblock Oscillatory integrals related to {C}arleson's theorem.
\newblock {\em Math. Res. Lett.}, 8(5-6):789--800, 2001.


\bibitem{W}
T.D. Wooley. Vinogradov's mean value theorem via efficient congruencing. Ann. of Math. 2 (2012), no. 175, 1575?1627.

\end{thebibliography}
\end{document}